\documentclass[10pt,a4paper]{article}

\usepackage{amsmath,amssymb,amsfonts,amsthm}
\usepackage{mathrsfs} 
\usepackage{bm} % bold math symbols
\usepackage{graphicx} 
\usepackage{xcolor}
\usepackage{booktabs}
\usepackage{caption}
\usepackage{hyperref}
\usepackage{geometry} 
\usepackage[table]{xcolor} % add this in the preamble
\usepackage{algorithm}
\usepackage{algorithmicx}
\usepackage{algpseudocode}

\usepackage{listings}

\theoremstyle{plain}
\newtheorem{theorem}{Theorem}[section] % numbered by section

\newtheorem{corollary}[theorem]{Corollary}

\theoremstyle{definition}
\newtheorem{definition}[theorem]{Definition}
\newtheorem{example}[theorem]{Example}

\theoremstyle{remark}
\newtheorem{remark}[theorem]{Remark}

\title{Representation of solutions to continuous and discrete first-order linear matrix equations with delay}
\author{Javad A. Asadzade\thanks{Department of Mathematics, Eastern Mediterranean University, North Cyprus, Turkey. Email: javad.asadzade@emu.edu.tr}
	\, and \, Nazim I. Mahmudov\thanks{Department of Mathematics, Eastern Mediterranean University, North Cyprus, Turkey; Research Center of Econophysics, Azerbaijan State University of Economics (UNEC), Baku, Azerbaijan. Email: nazim.mahmudov@emu.edu.tr}}
\date{} % remove date or add \today

\begin{document}
	
	\maketitle

\begin{abstract}
In this paper, we study continuous and discrete linear delay systems given respectively by
\[
\dot{X}(\xi) = A_0 X(\xi) + X(\xi)A_1 + B_0 X(\xi-\sigma) + X(\xi-\sigma)B_1 + G(\xi),
\]
and its discrete analogue
\[
X(u+1) = A_0 X(u) + X(u)A_1 + B_0 X(u-m) + X(u-m)B_1 + G(u),
\]
where \(A_0, A_1, B_0, B_1 \in \mathbb{R}^{d \times d}\) are constant noncommuting matrices, and \(\sigma>0\), \(m \in \mathbb{N}\) denote the delay parameters.  
The main objective is to generalize the classical results of \cite{diblik1, diblik2} and to provide explicit representations of the solutions. For this purpose, we present generalized delayed exponential-type systems for both continuous and discrete cases. This approach allows us to remove the restrictive commutativity conditions \(B_1G(\xi)=G(\xi)B_1\) and \(B_1\Psi(\xi)=\Psi(\xi)B_1\) imposed in \cite{diblik1, diblik2}, thus obtaining explicit solution formulas for more general classes of systems.
\end{abstract}

	
\section{Introduction}\label{sec1}

Linear matrix systems with delayed arguments constitute one of the fundamental areas of dynamical systems due to their wide applications in mathematics, control theory, signal processing, iterative learning, and stability analysis. The explicit representation of solutions to such systems plays a crucial role in their theoretical study and in the development of computational methods. Therefore, this topic has always attracted significant attention from researchers.

The pioneering work in this direction was obtained by Khusainov et al. \cite{Khusayanov1}, who investigated a linear matrix differential equation with delay and derived an explicit representation of its solution. Later, the fractional analogue of this problem was studied by Li and Wang \cite{Li1, Li2}, who established several results for single-delay fractional systems. Subsequently, Mahmudov \cite{Nazim} extended these ideas and developed a more general representation based on the delayed perturbation of the Mittag–Leffler function. 

Inspired by these studies, Huseynov and Mahmudov \cite{Ismail} introduced a delayed analogue of the three-parameter Mittag–Leffler function and applied it to the analysis of Caputo-type fractional time-delay differential equations with permutable matrices. In their work, the authors derived a new representation of solutions to homogeneous and nonhomogeneous fractional delay systems using the Laplace transform and the variation of constants formula, and further established existence and uniqueness results for nonlinear fractional delay systems via the contraction mapping principle.

In addition to these studies, analogous results for discrete, impulsive, and second-order systems have also been obtained. For instance, Diblík and Khusainov \cite{diblik5, diblik6} established representation formulas for discrete delayed systems with constant coefficients and commutative matrices, while later works by Diblík and Mencáková \cite{diblik4} and Elshenhab and Wang \cite{Elshenhab} extended these results to systems involving second-order differences and nonpermutable matrices. Further developments can be found in the works of Svoboda \cite{Svoboda}, who analyzed the representation of solutions for second-order differential systems with constant delays, and Khusainov and Shuklin \cite{Khusainov2003}, who considered autonomous delayed systems with permutation matrices. The results of Khusainov et al. \cite{Khusainov2008} generalized this framework to oscillating systems with pure delay, providing constructive forms of the fundamental solution. For discrete-time systems, Jin et al. \cite{Jin} investigated impulsive systems with nonpermutable matrices, while Liang et al. \cite{Liang2018} studied iterative learning control (ILC) for linear discrete systems with single delay, emphasizing the connection between representation formulas and control performance. Furthermore, Mahmudov \cite{Mahmudov2018, Mahmudov2024, Mahmudov2025, Mah} presented several representations for discrete linear systems with nonpermutable and multiple delays, including applications to iterative learning control. 

Parallel investigations for continuous-time and fractional-order systems can be found in the works of Chen \cite{Chen}, Liang et al. \cite{LiangORegan}, and Mahmudov \cite{Mahmudov2019, Mahmudov2022, Mahmudov2022b}, where various delayed Mittag–Leffler type matrix functions were employed to represent the solutions and analyze their stability properties. In the context of multi-delay and neutral-type systems, important contributions have been made by Medved and Pospíšil \cite{Medved2012, Medved2016}, as well as by Pospíšil \cite{Pospisil2012, Pospisil2017, Pospíšil}, who developed comprehensive frameworks for the representation and stability analysis of systems with permutable and nonpermutable matrix coefficients. More recently, Qiu and Wang \cite{QiuWang, Qiu2023} studied the representation of solutions for second-order and neutral-type delay differential equations, extending the theory to more general functional and fractal domains.

Overall, these developments collectively demonstrate that the representation of solutions for delayed matrix systems whether continuous, discrete, fractional, or multi-delayed remains a central problem of ongoing research interest. Moreover, the classical results obtained by Diblík \cite{diblik1, diblik2} for first-order differential and discrete systems with a single delay highlight the necessity for further exploration in this direction. Motivated by these works and the above-mentioned contributions, in this paper, we consider two-sided delayed matrix systems of the following form:
\begin{equation}\label{1}
\begin{cases}
\dot{X}(\xi) = A_0 X(\xi) + X(\xi) A_1 + B_0 X(\xi-\sigma) + X(\xi-\sigma) B_1 + G(\xi), & \xi \geq 0,\\[1mm]
X(\xi) = \Psi(\xi), & \xi \in [-\sigma,0],
\end{cases}
\end{equation}
where $\sigma > 0$ is the delay, $X(\xi)$ is the unknown $d \times d$ two-sided matrix, $A_0, A_1, B_0, B_1 \in \mathbb{R}^{d \times d}$ are given (generally noncommutative) constant matrices, and $G(\xi)$ is a given matrix-valued function. 

Similarly, we investigate the discrete-time counterpart of \eqref{1}, which is defined as
\begin{equation}\label{2}
\begin{cases}
X(u+1) = A_0 X(u) + X(u) A_1 + B_0 X(u-m) + X(u-m) B_1 + G(u), & u \in \mathbb{Z}_0^\infty,\\[1mm]
X(u) = \Psi(u), & u \in \mathbb{Z}_{-m}^0,
\end{cases}
\end{equation}
where $m \in \mathbb{N}$ is the discrete delay, and $X(u)$ denotes the unknown $d \times d$ \textbf{two-sided} matrix sequence.

Our main objective is to provide explicit representations of the solutions for these systems. To this end, we construct the fundamental solution, taking into account the noncommutativity of the coefficient matrices. This noncommutativity significantly increases the complexity of the problem, leading to more intricate solution structures.  

In this way, our approach extends the work of Diblík \cite{diblik1, diblik2} to both continuous and discrete time systems. Moreover, in particular cases such as the continuous system with \(A_0 = A_1 = \Theta\) or the discrete system with \(A_0 = I\), \(A_1 = \Theta\), and under additional conditions like \(B_1 G = G B_1\) and \(\Psi B_1 = B_1 \Psi\)—one can easily recover their results. Therefore, our method not only provides a more general framework but also allows the explicit representation of solutions without the restrictive assumptions mentioned above.

More recent studies \cite{Yang2025, Els, Dib2026, Els2} have addressed second-order, fractional-order, and multi-delay systems, obtaining parallel results by applying techniques similar to those of Diblík \cite{diblik1, diblik2} under analogous assumptions. However, it should be emphasized that these results still do not overcome the same limitations as in the classical case. In particular, their solution representations rely on commutativity conditions imposed on \(B_1\), \(G\), and \(\Psi\) (or their multi-delay analogues).  
Moreover, unlike our approach, these works consider systems without the delay-free terms \(A_0 X(t)\) and \(X(t) A_1\), yet they obtain analogous results. Therefore, the methodology presented in this paper overcomes these restrictions and enables stronger and more general solution representations.

Many realistic models involve noncommutative, two-sided matrices, where the order of multiplication fundamentally affects system dynamics. Developing constructive techniques that account for both delays and noncommutative, two-sided coefficients is therefore an important research direction.  

The structure of the paper is as follows. In Section~\ref{sec2}, we introduce the fundamental constructions for delayed matrix systems and define the auxiliary sequences \(Q_u\) used to build the fundamental matrix functions. Section~\ref{sec3} establishes the main properties of these fundamental solutions and extends these results to nonhomogeneous problems, providing general solution representations. Finally, Section~\ref{sec6} illustrates the applicability of our formulas with two-dimensional examples, highlighting the influence of noncommutative, two-sided matrices on the structure of solutions.

	\section{Fundamental Constructions}\label{sec2}

In this section, we introduce the auxiliary sequences and functions that serve as the foundation for explicit solutions of delayed matrix systems. We present a unified recursive construction of the matrices \(Q_u\) applicable to both continuous and discrete frameworks. The principal distinction between the two settings appears only in the final form of the fundamental function \(Z_D(\cdot)\).

\subsection{Auxiliary Sequences}

Let \(\{Q_u(\cdot)\}_{u\in\mathbb{N}}\) be a family of matrices recursively defined as
\begin{align}\label{Aux}
\begin{cases}
Q_{u+1}(l\delta) = A_0 Q_u(l\delta) + Q_u(l\delta) A_1 + B_0 Q_u((l-1)\delta) + Q_u((l-1)\delta) B_1,\\[0.5ex]
Q_0(l\delta) = Q_u(-\delta) = \Theta, \quad Q_1(0) = D, \quad u,l = 0,1,2,\dots,
\end{cases}
\end{align}
where \(\Theta\) and \(D\) denote the zero and a given \(d \times d\) matrices, respectively, and \(\delta\) represents the delay step, defined by
\[
\delta =
\begin{cases}
\sigma, & \text{for continuous-time systems},\\
m, & \text{for discrete-time systems}.
\end{cases}
\]

The general noncommutative structure of \(Q_{u+1}(k)\) is summarized in Table~\ref{tab:Q_table}.

\begin{table}[!ht]
\centering
\renewcommand{\arraystretch}{1.5} 
\setlength{\tabcolsep}{8pt} 
\rowcolors{2}{gray!10}{white}
\begin{tabular}{c|cccccc}
\hline
\rowcolor{gray!25}
$\mathcal{Q}_{u+1}(k)$ & $k=0$ & $k=\delta$ & $k=2\delta$ & $\cdots$ & $k=q\delta$ \\ \hline
$\mathcal{Q}_1(k)$ & $D$ & $\Theta$ & $\Theta$ & $\cdots$ & $\Theta$ \\ \hline
$\mathcal{Q}_2(k)$ & $A_0 D+DA_1$ & $B_0D+DB_1$ & $\Theta$ & $\cdots$ & $\Theta$ \\ \hline
$\mathcal{Q}_3(k)$ & $\sum_{n=0}^2 \binom{2}{n}A_0^n D A_1^{2-n}$ & $\cdots$ & $\sum_{n=0}^2 \binom{2}{n} B_0^n D B_1^{2-n}$ & $\cdots$ & $\Theta$ \\ \hline
$\vdots$ & $\vdots$ & $\vdots$ & $\vdots$ & $\ddots$ & $\Theta$ \\ \hline
$\mathcal{Q}_{q+1}(k)$ & $\sum_{n=0}^q \binom{q}{n}A_0^n D A_1^{q-n}$ & $\cdots$ & $\cdots$ & $\cdots$ & $\sum_{n=0}^q \binom{q}{n}B_0^n D B_1^{q-n}$ \\ \hline
\end{tabular}
\caption{Recursive structure of \(\mathcal{Q}_{u+1}(k)\)}
\label{tab:Q_table}
\end{table}

\subsection{Fundamental Functions}

\begin{definition}
For \(\xi \in \mathbb{R}\), the \emph{fundamental function} \(Z_D(\xi)\) for continuous-time systems is defined by
\begin{align}\label{Y_cont}
Z_D(\xi) =
\begin{cases}
\Theta, & \xi < 0,\\[0.5ex]
D, & \xi = 0,\\[0.5ex]
\displaystyle \sum_{r=0}^{\infty} Q_{r+1}(0) \frac{\xi^r}{r!} + \cdots + \sum_{r=n-1}^{\infty} Q_{r+1}((n-1)\sigma) \frac{(\xi-(n-1)\sigma)^r}{r!}, & (n-1)\sigma < \xi \le n\sigma.
\end{cases}
\end{align}
\end{definition}

Here, the coefficients of the first series correspond to \(k=0\) in Table~\ref{tab:Q_table}, the second series to \(k=\sigma\), and so forth. By regrouping terms, \(Z_D(\xi)\) can be written compactly as
\begin{align*}
Z_D(\xi) &= \sum_{k=0}^{\infty} \sum_{i=0}^{n-1} Q_{k+1}(i\sigma) \frac{(\xi - i\sigma)_+^k}{k!}, \quad (n-1)\sigma < \xi \le n\sigma,
\end{align*}
where \((\xi-i\sigma)_+^k = (\xi-i\sigma)^k\) if \(\xi \ge i\sigma\), and zero otherwise.

\begin{definition}
Let \(m \ge 1\) and let \(A_0,A_1,B_0,B_1,D \in \mathbb{R}^{d \times d}\) be constant matrices. The \emph{delayed perturbation of the discrete matrix exponential} is defined as
\begin{align}\label{Y_disc}
Z_D(u) = \sum_{i=0}^{\lfloor \frac{u+m}{m+1} \rfloor} Q_{u-(i-1)m}(i), \quad u \in \mathbb{Z}_0^\infty.
\end{align}
\end{definition}

If \(l = \lfloor \frac{u+m}{m+1} \rfloor\), then by Lemma 3.4 in \cite{Mah}, the discrete fundamental solution can equivalently be expressed as
\begin{align}\label{Y_disc1}
Z_D(u) =
\begin{cases}
\Theta, & u \in \mathbb{Z}_{-\infty}^{-m-1},\\[0.5ex]
\displaystyle \sum_{n=0}^{u+m-1} \binom{u+m-1}{n} A_0^n D A_1^{u+m-1-n} + \sum_{i=1}^{l} Q_{u-(i-1)m}(i), & u \in \mathbb{Z}^{l(m+1)}_{(l-1)(m+1)+1}.
\end{cases}
\end{align}

	\section{Main results}\label{sec3}
	This section examines the solution of \eqref{1} and \eqref{2}.
	\subsection{Continious-time case}
	This subsection is devoted to the study of the following homogeneous matrix delay problem:
	\begin{equation}\label{q2}
		\begin{cases}
			\dot{Z}_D(\xi) =A_0 Z_D(\xi) + Z_D(\xi) A_1+ B_0 Z_D(\xi-\sigma) + Z_D(\xi-\sigma) B_1, & \xi \in (0, \infty),\\[1mm]
			Z_D(\xi) = D, & \xi \in [- \sigma, 0],
		\end{cases}
	\end{equation}
	where the matrices $A_0, \, A_1,\,B_0$ and $B_1$ do not commute each other.
	
	This section aims to derive an explicit representation of the solution to \eqref{q2}, which is established in the theorem below.
	
	\begin{theorem}\label{t1}
		$ Z_D(\xi)$ is a solution of 
		\begin{align}\label{q1}
			\dot{Z}_D(\xi)=A_0 Z_D(\xi)+Z_D(\xi)A_1+B_0 Z_D(\xi-\sigma)+Z_D(\xi-\sigma)B_1.
		\end{align}
	\end{theorem}
	\begin{proof}
   We show that $Z_D(\xi)$ satisfies differential equation \eqref{q1} on each interval $\xi\in  (\xi_{u-1},\xi_u]$, for $u=1,2,3,\dots$. 
    
		\noindent\textbf{(i)} For \textbf{$0 < t \leq \sigma$:}  
\[
\dot{Z}_D(\xi) = \sum_{i=1}^{\infty} Q_{i+1}(0) \frac{\xi^{i-1}}{(i-1)!} 
= \sum_{i=0}^{\infty} Q_{i+2}(0) \frac{\xi^i}{i!}.
\]  
Using the recursion for $Q_{i+2}(0)$, we get
\[
\dot{Z}_D(\xi) = \sum_{i=0}^{\infty} \big(A_0 Q_{i+1}(0) + Q_{i+1}(0) A_1 + B_0 Q_{i+1}(-\sigma) + Q_{i+1}(-\sigma) B_1\big) \frac{\xi^i}{i!}.
\]  
Since $Q_{i+1}(-\sigma) = \Theta$ for all $i$, it follows that
\[
\dot{Z}_D(\xi) = A_0 Z_D(\xi) + Z_D(\xi) A_1.
\]  

\noindent\textbf{(ii)} For \textbf{$\sigma < t \leq 2\sigma$:}  

\begin{align*}
  &\qquad  \dot{Z}_D(\xi) = \sum_{i=1}^{\infty} Q_{i+1}(0) \frac{\xi^{i-1}}{(i-1)!}+\sum_{i=1}^{\infty} Q_{i+1}(\sigma) \frac{(\xi -\sigma)^{i-1}}{(i-1)!} 
= \sum_{i=0}^{\infty} Q_{i+2}(0) \frac{\xi^i}{i!}+\sum_{i=0}^{\infty} Q_{i+2}(\sigma) \frac{(\xi -\sigma)^{i}}{i!} \\
&= \sum_{i=0}^{\infty}\left( A_0 Q_{i+1}(0)+Q_{i+1}(0)A_1\right) \frac{\xi^i}{i!}+\sum_{i=0}^{\infty} \left(A_0 Q_{i+1}(\sigma)+Q_{i+1}A_1+B_0 Q_{i+1}(0)+Q_{i+1}(0)B_1\right) \frac{(\xi -\sigma)^{i}}{i!} \\
&=A_0\left(\sum_{i=0}^{\infty}Q_{i+1}(0) \frac{\xi^i}{i!}+\sum_{i=1}^{\infty}  Q_{i+1}(\sigma)\frac{(\xi -\sigma)^{i}}{i!}\right)+\left(\sum_{i=0}^{\infty}Q_{i+1}(0) \frac{\xi^i}{i!}+\sum_{i=1}^{\infty}  Q_{i+1}(\sigma)\frac{(\xi -\sigma)^{i}}{i!}\right)A_1\\
&+B_0 \left(\sum_{i=0}^{\infty}Q_{i+1}(0) \frac{(\xi -\sigma)^i}{i!}\right)+ \left(\sum_{i=0}^{\infty}Q_{i+1}(0) \frac{(\xi -\sigma)^i}{i!}\right)B_1\\
&=A_0Z_D(\xi)+Z_D(\xi)A_1+B_0 Z_D(\xi-\sigma)+Z_D(\xi-\sigma)B_1
\end{align*}
		
		\noindent\textbf{(iii)} For $
        u\sigma<\xi \leq (u+1)\sigma$, we have
	
		\begin{align*} \dot{Z}_D(\xi) &= \sum_{k=l}^{\infty}\;\sum_{l=0}^{u} Q_{k+1}(l\sigma)\frac{(\xi-l\sigma)_+^{k-1}}{(k-1)!}\\ &=\sum_{k=l}^{\infty}\;\sum_{l=0}^{u}\left(A_0 Q_{k}(l\sigma)+Q_{k}(l\sigma)A_1+B_0Q_{k}((l-1)\sigma )+Q_{k}((l-1)\sigma )B_1\right)\frac{(\xi-l\sigma)_+^{k-1}}{(k-1)!}\\ &=\sum_{k=l}^{\infty}\;\sum_{l=0}^{u}A_0 Q_{k+1}(l\sigma)\frac{(\xi-l\sigma)_+^{k}}{k!}+\sum_{k=l}^{\infty}\;\sum_{l=0}^{u} Q_{k+1}(l\sigma)A_1\frac{(\xi-l\sigma)_+^{k}}{k!}\\ &+\sum_{k=l}^{\infty}\;\sum_{l=0}^{u-1}B_0 Q_{k+1}(l\sigma)\frac{(\xi-(l+1)\sigma )_+^{k}}{k!}+\sum_{k=l}^{\infty}\;\sum_{l=0}^{u-1} Q_{k+1}(l\sigma)B_1\frac{(\xi-(l+1)\sigma )_+^{k}}{k!}\\ &=A_0Z_D(\xi)+Z_D(\xi)A_1+B_0Z_D(\xi-\sigma)+Z_D(\xi-\sigma)B_1. \end{align*}
	\end{proof}
	From Theorem \ref{t1} and the first two formulas in (4), it follows immediately 
	\begin{theorem}
		Let $A_0,A_1,B_0$ and $B_1$ be non-commutative constant $d\times d$ matrices. Then the matrix $Z_D(\xi)$ is the unique solution of the initial Eqs. \eqref{q2}.
	\end{theorem}
    \begin{corollary}
        If $A_1=B_1=\Theta$, then the unique solution of \eqref{q2} is given by
        \begin{align*}
            \sum_{k=0}^{\infty}\;\sum_{i=0}^{u-1} P_{k+1}(i\sigma)\,\frac{(\xi-i\sigma)_+^{k}}{k!}, 
    \quad (u-1)\sigma < \xi \leq u\sigma,
        \end{align*}
        where 
        
        \begin{align*}
          &  P_{u+1}(l\sigma)=A_0 P_{u}(l\sigma)+B_0 P_{u}((l-1)\sigma),\\
       & P_0(l\delta) = P_u(-\delta) = \Theta, \quad P_1(0) = D, \quad u,l=0,1,2,\dots
        \end{align*}
    \end{corollary}
	\begin{corollary}\label{col1}
		If $A_0 = A_1=\Theta$, then the unique solution of \eqref{q2} is given by
		\[
		Z_D(\xi) = \sum_{j=0}^u\sum_{l=0}^{j} \binom{u}{l} B_0^{u-l}D B_1^l \frac{(\xi-(j-1)\sigma)^j}{j!} , 
		\quad \text{for } \xi \in [(u-1)\sigma,\, u \sigma).
		\]
	\end{corollary}
	\begin{remark}
	    Note that, if we take $D=I$ in corollary \ref{col1}, then the above result coincides by \cite{diblik1}.
	\end{remark}
	\begin{corollary}
		If $A_0 = A_1=B_1 = \Theta$, and $D=I$, then the unique solution of \eqref{q2} is
		\[
		Z_I(\xi) = \sum_{j=0}^u B_0^j\frac{(\xi-(j-1)\sigma)^j}{j!}  = e^{B_0 (\xi-\sigma)}, 
		\quad \text{for } \xi \in [(u-1)\sigma,\, u \sigma).
		\]
	\end{corollary}

    	\begin{theorem}\label{teo1}
Assume that $G(\xi)$ is continuous on $[0,\infty)$ and the initial function $\Psi:[-\sigma,0]\to\mathbb{R}^d$ is continuous. Then the solution of \eqref{1} is given by
\begin{equation}\label{eq:X13}
X(\xi)
= Z_{\Psi(0)}(\xi)
+ \int_{-\sigma}^0 Z_{B_0\Psi(s)+\Psi(s)B_1}(\xi-\sigma-s)\,ds
+ \int_0^\xi Z_{G(s)}(\xi-s)\,ds.
\end{equation}
\end{theorem}
\begin{proof}
Let
\[
X(\xi)
:= Z_{\Psi(0)}(\xi)
+ \int_{-\sigma}^0 Z_{B_0\Psi(s)+\Psi(s)B_1}(\xi-\sigma-s)\,ds
+ \int_0^\xi Z_{G(s)}(\xi-\sigma-s)\,ds
\]
be the candidate solution. We show that \(X\) satisfies the history and the differential equation.

\medskip\noindent\textbf{(1) History.}
If \(-\sigma\le\xi<0\) then for every \(D\) the fundamental function satisfies
\(Z_D(\xi)=\Theta\) (by definition of \(Z_D\) on \((-\infty,0)\)), hence the first term equals \(\Theta\), the third integral vanishes (upper limit \(\xi<0\)), and in the middle integral
\(\xi-\sigma-s<0\) for all \(s\in[-\sigma,0]\) so the integrand is \(\Theta\). Therefore \(X(\xi)=\Psi(\xi)\) for \(\xi\in[-\sigma,0)\).
At \(\xi=0\) we have \(Z_{\Psi(0)}(0)=\Psi(0)\) and the two integrals vanish, so \(X(0)=\Psi(0)\). Thus \(X\) matches the prescribed initial function on \([-\sigma,0]\).

\medskip\noindent\textbf{(2) Differentiation for \(\xi>0\).}
Fix \(\xi>0\). By Theorem \ref{t1} each \(Z_D\) is differentiable on \((0,\infty)\) and satisfies the homogeneous identity
\begin{equation}\label{ZD-ode}
\dot Z_D(\xi)=A_0 Z_D(\xi)+Z_D(\xi)A_1+B_0 Z_D(\xi-\sigma)+Z_D(\xi-\sigma)B_1,
\qquad \xi>0.
\end{equation}
Moreover, for fixed \(\xi\) the truncated-power representation of \(Z_D\) yields only finitely many nonzero polynomial terms, so differentiation may be interchanged with the finite sums and with the integrals (Leibniz rule applies). Differentiate \(X\) with respect to \(\xi\):
\[
\begin{split}
\dot X(\xi)
&= \dot Z_{\Psi(0)}(\xi)
+ \int_{-\sigma}^0 \frac{\partial}{\partial\xi}Z_{B_0\Psi(s)+\Psi(s)B_1}(\xi-\sigma-s)\,ds + \int_0^\xi \frac{\partial}{\partial\xi} Z_{G(s)}(\xi-s)\,ds
\;+\; Z_{G(\xi)}(0)\\
&= A_0 Z_{\Psi(0)}(\xi)+Z_{\Psi(0)}(\xi)A_1+B_0 Z_{\Psi(0)}(\xi-\sigma)+Z_{\Psi(0)}(\xi-\sigma)B_1\\
& + \int_{-\sigma}^0 \Big[ A_0 Z_{B_0\Psi(s)+\Psi(s)B_1}(\xi-\sigma-s) + Z_{B_0\Psi(s)+\Psi(s)B_1}(\xi-\sigma-s) A_1 \\
& + B_0 Z_{B_0\Psi(s)+\Psi(s)B_1}(\xi-2\sigma-s) + Z_{B_0\Psi(s)+\Psi(s)B_1}(\xi-2\sigma-s) B_1 \Big] ds\\
& + \int_0^\xi \Big[ A_0 Z_{G(s)}(\xi-s) + Z_{G(s)}(\xi-s) A_1 + B_0 Z_{G(s)}(\xi-\sigma-s) + Z_{G(s)}(\xi-\sigma-s) B_1 \Big] ds  + G(\xi)\\
&= A_0 \underbrace{\Big(Z_{\Psi(0)}(\xi) + \int_{-\sigma}^0 Z_{B_0\Psi(s)+\Psi(s)B_1}(\xi-\sigma-s) ds + \int_0^\xi Z_{G(s)}(\xi-s) ds \Big)}_{X(\xi)} \\
& + \underbrace{\Big(Z_{\Psi(0)}(\xi) + \int_{-\sigma}^0 Z_{B_0\Psi(s)+\Psi(s)B_1}(\xi-\sigma-s) ds + \int_0^\xi Z_{G(s)}(\xi-s) ds \Big)}_{X(\xi)} A_1\\
& + B_0 \underbrace{\Big(Z_{\Psi(0)}(\xi-\sigma) + \int_{-\sigma}^0 Z_{B_0\Psi(s)+\Psi(s)B_1}(\xi-2\sigma-s) ds + \int_0^{\xi-\sigma} Z_{G(s)}(\xi-\sigma-s) ds \Big)}_{X(\xi-\sigma)}\\
& + \underbrace{\Big(Z_{\Psi(0)}(\xi-\sigma) + \int_{-\sigma}^0 Z_{B_0\Psi(s)+\Psi(s)B_1}(\xi-2\sigma-s) ds + \int_0^{\xi-\sigma} Z_{G(s)}(\xi-\sigma-s) ds \Big)}_{X(\xi-\sigma)} B_1 + G(\xi)\\
&= A_0 X(\xi) + X(\xi) A_1 + B_0 X(\xi-\sigma) + X(\xi-\sigma) B_1 + G(\xi).
\end{split}
\]

\medskip\noindent\textbf{(3) Uniqueness.}
Uniqueness follows from standard linear theory for delay equations (or a Grönwall-type argument in the history space): any solution of the homogeneous problem with zero history is identically zero. Thus the constructed \(X\) is the unique solution of \eqref{1}.
\end{proof}
\begin{corollary}
Consider Eq.~\eqref{1} with $A_0 = A_1 = \Theta$ (zero matrices), $B_0, B_1 \in \mathbb{R}^{d\times d}$ constant, $G:[0,\infty)\to \mathbb{R}^{d\times d}$ continuous, and $\Psi:[-\sigma,0]\to \mathbb{R}^{d\times d}$ continuous. Then the unique solution of
\[
\dot X(\xi) = B_0 X(\xi-\sigma) + X(\xi-\sigma) B_1 + G(\xi), \quad \xi>0,
\]
with initial history $X(\xi) = \Psi(\xi)$ for $\xi \in [-\sigma,0]$, is
\[
X(\xi) = Z_{\Psi(0)}(\xi) + \int_{-\sigma}^0 Z_{B_0 \Psi(s)+\Psi(s) B_1}(\xi-\sigma-s)\, ds
+ \int_0^\xi Z_{G(s)}(\xi-\sigma-s)\, ds,
\]
where $Z_D(\xi)$ denotes the fundamental solution of the homogeneous problem
\[
\dot Z_D(\xi) = B_0 Z_D(\xi-\sigma) + Z_D(\xi-\sigma) B_1, \quad Z_D(\xi) = D \text{ for } \xi \in [-\sigma,0].
\]
\end{corollary}

\begin{corollary}
Let $A_0 = A_1 = \Theta$, $D = I$, and assume that $B_1 G(\xi) = G(\xi) B_1$ and $B_1 \Psi(\xi) = \Psi(\xi) B_1$ for all $\xi$, with $\Psi \in C^1([-\sigma,0],\mathbb{R}^{d\times d})$ and $G \in C([0,\infty),\mathbb{R}^{d\times d})$. Then the unique solution of
\[
\dot X(\xi) = B_0 X(\xi-\sigma) + X(\xi-\sigma) B_1 + G(\xi), \quad \xi>0,
\]
with initial history $X(\xi) = \Psi(\xi)$ for $\xi \in [-\sigma,0]$, is
\[
X(\xi) = Z_I(\xi)\Psi(0) 
+ \int_{-\sigma}^0 Z_I(\xi-\sigma-s)\Psi^{\prime}(s)\, ds
+ \int_0^\xi Z_I(\xi-\sigma-s)G(s)\, ds,
\]
where $Z_I(\xi)$ denotes the fundamental solution of
\[
\dot Z_I(\xi) = B_0 Z_I(\xi-\sigma) + Z_I(\xi-\sigma) B_1, \quad Z_I(\xi) = I \text{ for } \xi \in [-\sigma,0].
\]
\end{corollary}

\subsection{Discrete-time case}

We consider the discrete-time delayed matrix difference equation
\begin{equation}\label{eq:discrete_hom}
\begin{cases}
X(u+1) = A_0 X(u) + X(u) A_1 + B_0 X(u-m) + X(u-m) B_1, & u \in \mathbb{Z}_0^\infty,\\[0.5ex]
X(u) = Q_{u+m}(0), & u \in \mathbb{Z}_{-m}^0,\\
X(u) = \Theta, & u \in \mathbb{Z}_{-\infty}^{-m-1},
\end{cases}
\end{equation}
where $A_0, A_1, B_0, B_1$ are fixed $d\times d$ non-commutative matrices, $D$ is any matrix, $m\in \mathbb{N}$ is the discrete delay.

\begin{theorem}
 $Z_D(u)$  is the fundamental solution of \eqref{eq:discrete_hom}.
\end{theorem}

\begin{proof}
\textbf{Step 1: Initial conditions.}  
By construction, $Z_D(u) = \Theta$ for $u \in \mathbb{Z}_{-\infty}^{-m-1}$, and $Z_D(u) = Q_{u+m}(0)$ for $u \in \mathbb{Z}_{-m}^0$. Thus, $Z_D(u)$ satisfies the prescribed history.

\medskip
\textbf{Step 2: Verification of the difference equation.}  
For $u\ge 0$:
\begin{align*}
Z_D(u+1) &= \sum_{i=0}^{\lfloor \frac{u+1+m}{m+1}\rfloor} Q_{u+1-(i-1)m}(i)\\
&= \sum_{i=0}^{\lfloor \frac{u+m}{m+1}\rfloor} A_0 Q_{u-(i-1)m}(i) + \sum_{i=0}^{\lfloor \frac{u+m}{m+1}\rfloor} Q_{u-(i-1)m}(i) A_1 \\
& + \sum_{i=1}^{\lfloor \frac{u+1+m}{m+1}\rfloor} B_0 Q_{u-(i-1)m}(i-1) + \sum_{i=1}^{\lfloor \frac{u+1+m}{m+1}\rfloor} Q_{u-(i-1)m}(i-1) B_1\\
&= A_0 Z_D(u) + Z_D(u) A_1 + B_0 Z_D(u-m) + Z_D(u-m) B_1.
\end{align*}

Thus $Z_D(u)$ satisfies \eqref{eq:discrete_hom} for all $u \ge 0$, and is the fundamental solution.
\end{proof}

Consider the initial value problem
\begin{equation}\label{eq:discrete_nonhom}
\begin{cases}
X(u+1) = A_0 X(u) + X(u) A_1 + B_0 X(u-m) + X(u-m) B_1 + G(u), & u \in \mathbb{Z}_0^\infty,\\
X(u) = \Psi(u), & u \in \mathbb{Z}_{-m}^0,
\end{cases}
\end{equation}
where $G:\mathbb{Z}_0^\infty \to \mathbb{R}^{d\times d}$ is a given forcing sequence, $\Psi:\mathbb{Z}_{-m}^0\to \mathbb{R}^{d\times d}$ is the initial history, and $m \in \mathbb{N}$ is the discrete delay.

\begin{theorem}\label{thm:discrete_nonhom}
Let $Z_D(u)$ be the fundamental solution of the \eqref{eq:discrete_hom}. Then the solution of \eqref{eq:discrete_nonhom} is given by
\begin{equation}\label{eq:solution_discrete_nonhom}
X(u) = Z_{\Psi(0)}(u-m) + \sum_{r=-m+1}^{0} Z_{B_0 \Psi(r) + \Psi(r) B_1}(u-1-2m-r) + \sum_{r=1}^{u} Z_{G(r-1)}(u-m-r).
\end{equation}
\end{theorem}

\begin{proof}
\textbf{Step 1:}  
For $u \in \mathbb{Z}_{-m}^0$, all terms with $Z_D(u-k)$ for $u-k \ge 0$ vanish because the fundamental solution is zero for negative indices beyond the prescribed history. The remaining terms reduce exactly to $\Psi(u)$. Hence, $X(u) = \Psi(u)$ for $u \in \mathbb{Z}_{-m}^0$, satisfying the initial condition.

\medskip
\textbf{Step 2: }  
For $u \ge 0$, using linearity of the forward difference and the property of the fundamental solution $Z_D$:
\[
Z_D(u+1) = A_0 Z_D(u) + Z_D(u) A_1 + B_0 Z_D(u-m) + Z_D(u-m) B_1,
\]
we compute
\begin{align*}
 X(u+1) &=  Z_{\Psi(0)}(u+1-m) + \sum_{r=-m+1}^{0} Z_{B_0 \Psi(r) + \Psi(r) B_1}(u-2m-r) + \sum_{r=1}^{u+1} Z_{G(r-1)}(u+1-m-r)\\
 &=A_0 Z_{\Psi(0)}(u-m)+Z_{\Psi(0)}(u-m)A_1+ B_0 Z_{\Psi(0)}(u-2m)+Z_{\Psi(0)}(u-2m)B_1\\
 &+   \sum_{r=-m+1}^{0} A_0 Z_{B_0 \Psi(r) + \Psi(r) B_1}(u-1-2m-r)+ \sum_{r=-m+1}^{0} Z_{B_0 \Psi(r) + \Psi(r) B_1}(u-1-2m-r)A_1\\
 &+   \sum_{r=-m+1}^{0} B_0 Z_{B_0 \Psi(r) + \Psi(r) B_1}(u-1-3m-r)+ \sum_{r=-m+1}^{0} Z_{B_0 \Psi(r) + \Psi(r) B_1}(u-1-3m-r)B_1\\
 &+ \sum_{r=1}^{u} A_0 Z_{G(r-1)}(u-m-r)+\sum_{r=1}^{u} Z_{G(r-1)}(u-m-r) A_1 \\
  &+ \sum_{r=1}^{u} B_0 Z_{G(r-1)}(u-2m-r)+\sum_{r=1}^{u} Z_{G(r-1)}(u-2m-r) B_1 +Z_{G(u)}(-m)\\
&= A_0 X(u) + X(u) A_1 + B_0 X(u-m) + X(u-m) B_1 + G(u).
\end{align*}
Thus, $X(u)$ satisfies the difference equation \eqref{eq:discrete_nonhom}.

\medskip
\textbf{Step 3: }  
Uniqueness follows from linearity and the zero solution property of the homogeneous system: any solution with zero history and zero forcing is identically zero. Hence, the constructed $X(u)$ is the unique solution.
\end{proof}

\begin{corollary}
Suppose $A_0 = D = I$, $A_1 = \Theta$, and the matrices $B_1$ and $G(u)$ commute for all $u \in \mathbb{Z}_{0}^{\infty}$, i.e., $B_1 G(u) = G(u) B_1$, and $B_1$ and $\Psi(u)$ commute for all $u \in \mathbb{Z}_{-m}^{0}$, i.e., $B_1 \Psi(u) = \Psi(u) B_1$. Then the solution of \eqref{2} is given by
\begin{equation}
X(u) = Z_{I}(u)\, \Psi(-m) + \sum_{r=-m+1}^{0} Z_{I}(u-m-r)\, \Delta \Psi(r-1) + \sum_{r=1}^{u} Z_{I}(u-m-r)\, G(r-1),
\end{equation}
where $\Delta \Psi(r-1) := \Psi(r) - \Psi(r-1)$.
\end{corollary}

\section{Illustrative Example}\label{sec6} 

\begin{example} \label{ex1}
Consider~\eqref{1} with $\sigma=1$ and $d=2$. Let the coefficient matrices and functions be defined as follows:
\[
A_0=\begin{pmatrix}1&0\\0&0\end{pmatrix},\quad  
A_1=\begin{pmatrix}0&0\\0&1\end{pmatrix},\quad  
B_0=\begin{pmatrix}1&0\\0&-1\end{pmatrix},\quad  
B_1=\begin{pmatrix}-1&0\\0&1\end{pmatrix},\quad  
G(\xi)=\xi I_2,\quad  
\Psi(\xi)=I_2.
\]

Under these assumptions, the expression for $X(\xi)$ becomes
\begin{equation}
X(\xi) = Z_{I_2}(\xi) + \int_{-1}^0 Z_{B_0+B_1}(\xi-1-s)\,ds + \int_0^\xi Z_{G(s)}(\xi-s)\,ds.
\end{equation}

Since $B_0 + B_1 = \Theta$, the middle term in the above expression vanishes. Hence, by Theorem~\ref{teo1}, the solution of~\eqref{1} simplifies to
\begin{equation}
X(\xi) = Z_{I_2}(\xi) + \int_0^\xi Z_{G(s)}(\xi-s)\,ds.
\end{equation}

Next, we compute $Z_{I_2}(\xi)$. According to the definition of the fundamental matrix, we have
\begin{align}
Z_{I_2}(\xi) = 
\begin{cases}
\Theta, & \xi < 0,\\[0.5ex]
I_2, & \xi = 0,\\[0.5ex]
e^{\xi}I_2, & \xi>0.
\end{cases}
\end{align}

Similarly, for $Z_{G(s)}(\xi)$, we obtain
\begin{align}
Z_{G(s)}(\xi) = 
\begin{cases}
\Theta, & \xi < 0,\\[0.5ex]
G(s), & \xi=0,\\[0.5ex]
s e^{\xi}I_2, & \xi>0.
\end{cases}
\end{align}

Substituting these expressions into the formula for $X(\xi)$ and performing the integration gives
\begin{align*}
X(\xi)=
\begin{cases}
I_2, & -1\leq \xi\leq 0,\\[0.5ex]
\left(2e^{\xi}-\xi-1\right)I_2, & \xi>0.
\end{cases}
\end{align*}

Thus, the explicit form of $X(\xi)$ demonstrates how the solution evolves across the regions $-1\leq \xi\leq0$, and $\xi>0$, highlighting the smooth transition governed as Figure \ref{Fiq1}.
\begin{figure}
    \centering
    \includegraphics[width=0.8\linewidth]{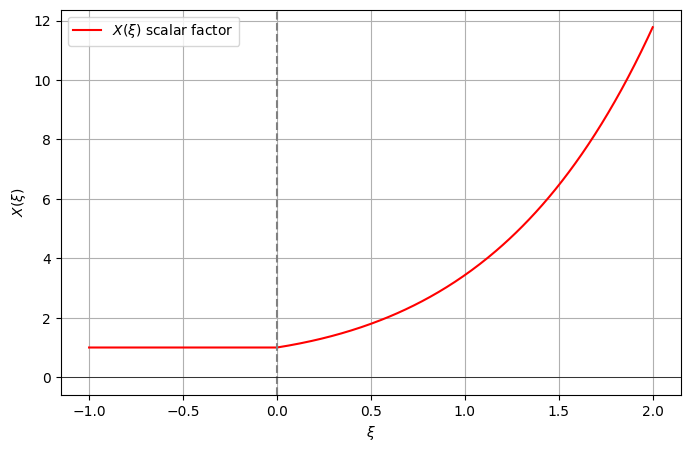}
    \caption{Representation of $X_{ii}(\xi),\, i=1,2$}
    \label{Fiq1}
\end{figure}
\end{example}

\begin{example} \label{ex2}
Consider the discrete-time system~\eqref{2} with the same coefficient matrices and functions as in Example~\ref{ex1}. Then the solution of the discrete-time system is given by

\begin{equation}
X(u) = Z_{I_2}(u-1) + \sum_{r=1}^{u} Z_{G(r-1)}(u-1-r).
\end{equation}

Let us compute $X(u)$ explicitly by using a discrete exponential approximation for $Z_D(u)$:

\[
Z_{I_2}(u-1) = 
\begin{cases}
\Theta, & u\in \mathbb{Z}_{-\infty}^{0},\\
I_2, & u\in \mathbb{Z}_1^{\infty}.
\end{cases}
\]
and 
\[
Z_{G(r-1)}(u-1-r) = 
\begin{cases}
\Theta, & u\in \mathbb{Z}_{-\infty}^{r},\\
(r-1)I_2, & u\in \mathbb{Z}_{r+1}^{\infty}.
\end{cases}
\]
Then 
\begin{align*}
X(u)=\begin{cases}
  I_2, & u\in \mathbb{Z}_{-1}^1,\\
    \frac{u^2-3u+4}{2}  I_2,& u\in \mathbb{Z}_{2}^{\infty}.
\end{cases}
\end{align*}

This illustrates the step-by-step evolution of the discrete-time solution $X(u)$ under the given matrices and functions in Figure \ref{Fig2}.

\begin{figure}
    \centering
    \includegraphics[width=0.8\linewidth]{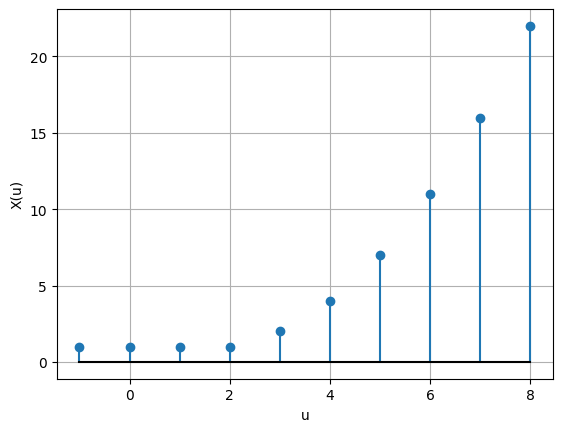}
    \caption{Representation of $X_{ii}(u),\, i=1,2$}
    \label{Fig2}
\end{figure}
\end{example}

\end{document}